\documentclass[a4paper, 12pt, oneside]{amsart}

\usepackage[top=35mm,bottom=15mm,left=20mm,right=20mm]{geometry}
\usepackage{mathrsfs}
\usepackage{mathtools}
\usepackage{amsfonts}
\usepackage{amssymb}
\usepackage{amsxtra}
\usepackage{color}
\usepackage{dsfont}
\usepackage{bbm}

\usepackage[english,polish]{babel}
\usepackage[colorlinks,citecolor=blue,urlcolor=blue,bookmarks=true]{hyperref}

\newcommand{\NN}{\mathbb{N}}
\newcommand{\R}{\mathbb{R}}

\newcommand{\Ss}{\mathbb{S}}

\newcommand{\Per}{\mathrm{Per}}

\newcommand{\ud}{\mathrm{d}}

\newcommand{\Rd}{\mathbb{R}^d}
\newcommand{\eps}{\varepsilon}

\renewcommand{\Cap}{\mathrm{Cap}}

\newcommand{\Wp}{W_p^{\nu}}

\newcommand{\dist}{\operatorname{dist}}
\newcommand{\supp}{\operatorname{supp}}

\newcommand{\pl}[1]{\foreignlanguage{polish}{#1}}

\newtheorem{theorem}{Theorem}
\numberwithin{theorem}{section}
\newtheorem{proposition}[theorem]{Proposition}
\newtheorem{lemma}[theorem]{Lemma}
\newtheorem{corollary}[theorem]{Corollary}

\theoremstyle{definition}
\newtheorem{example}[theorem]{Example}
\newtheorem{definition}{Definition}
\numberwithin{definition}{section}

\numberwithin{remark}{section}

\numberwithin{equation}{section}

\title{Hardy inequalities and nonlocal capacity}
\author{Tomasz Grzywny}
\address{Wroc{\l}aw University of Science and Technology,
Faculty of Pure and Applied Mathematics\\
	Wyb. \pl{Wyspia\'{n}skiego} 27,
	50-370 \pl{Wroc\l{}aw}, Poland}
\email{tomasz.grzywny@pwr.edu.pl}
\author[J.~Lenczewska]{Julia Lenczewska}
\address{Faculty of Pure and Applied Mathematics, Wroc\l{}aw University of Science and Technology, Wyb. Wyspia\'nskiego 27, 50-370 Wroc\l{}aw, Poland.}
\email{julia.lenczewska@pwr.edu.pl}
\thanks{The authors were partially supported by National Science Centre (Poland), grant no.\ 2019/33/B/ST1/02494}

				

\begin{document}
\selectlanguage{english}

\begin{abstract}
In this article, we introduce and study capacities related to nonlocal Sobolev spaces, with focus on spaces corresponding to zero-order nonlocal operators. In particular, we prove Hardy-type inequalities to obtain Sobolev embeddings and use them to estimate the nonlocal capacities of a ball. 
\end{abstract}

\maketitle

\section*{Introduction}
The notion of capacity has its origin in the Coulomb's law in electrostatics and Newton's law of universal gravitation.    
Around 1923-1925 Wiener introduced the modern theory of capacity to mathematics. Now it is broadly used in studying various problems arising from partial differential equations, potential theory, geometric harmonic analysis and mathematical physics.

Capacities generated by classical Sobolev spaces in $\Rd$ were studied in \cite[Section 4.7]{MR1158660}. In the monograph \cite[Section 2.1]{MR2778606} the authors considered capacities related to regular Dirichlet forms on $L^2$. In the fractional setting, Shi and Xiao \cite{MR3518675, MR3591352} studied the capacity generated by the fractional Sobolev space
\begin{equation}\label{frac-Sob-sp}
W^{s}_p = \left\{f : \|f\|_{W_p^s} := \|f\|_{p} + \left(\int_{\Rd } \int_{\Rd} |f(x+h)-f(x)|^p |h|^{-d-sp}
\, \ud x \ud h\right)^{1/p} <\infty
	\right\},
\end{equation}
where $s\in(0,1)$ and $p\in[1,\infty)$. The logarithmic Sobolev capacity, i.e. the capacity related to the seminorm

$$
[f]_{W_p^{\log , \gamma}} = 
\left(\int_{|h|<1} \int_{\Rd} |f(x+h)-f(x)|^p \left(\log\frac{e}{|h|}\right)^{\gamma p -1} \frac{ \ud x \ud h}{|h|^d} \right)^{1/p},
$$
where $\gamma \in(0,\infty)$ and $p \in [1,\infty)$, has been recently studied by Liu, Wu, Xiao, and Yuan \cite{MR4313960}.

Our goal is to introduce and study the capacities generated by nonlocal Sobolev spaces. In particular, we shall prove a Hardy inequality and use it to obtain a nonlocal Sobolev embedding and in turn, estimates of nonlocal capacity of a ball. 

Let $p\in[1,\infty)$ and let $\nu$ be a
Borel measure on $\Rd$ satisfying
\begin{equation}\label{integrability-cap}
\int_{\Rd} \left(1 \wedge |x|^p \right)\nu (\ud x) < \infty.
\end{equation}
The classic example of such measure is $\nu(\ud x) =|x|^{-d-sp}\, \ud x$ from \eqref{frac-Sob-sp}.
We define the nonlocal Sobolev space
$$
W_p^{\nu} = \{ f: \|f\|_{W^{\nu}_p} := \|f\|_{p} + \left[f\right]_{W_{p}^{\nu}} <\infty\}, 
$$
where
$$
\left[f\right]_{W_{p}^{
\nu}} := \left(\int_{\Rd} \int_{\Rd} |f(x+h)-f(x)|^p \, \ud x
\, \nu (\ud h) \right)^{1/p}.
$$
The $W^{\nu}_p$-capacity of an arbitrary set $E \subset \Rd$ is defined by
\begin{equation}\label{Capnu}
\Cap_{\nu,p} (E) := \inf \{ \|f\|_{W^{\nu}_p}: f \in W_p^{\nu} \,\, \textnormal{and} \,\, E \subset \operatorname{int}(\{f\geq 1\}) \}.
\end{equation}

We are interested in obtaining a Sobolev embedding, which will be used to estimate the nonlocal capacity of a ball. To this end, we will prove 
the following 
Hardy inequality: if $\nu$ is absolutely continuous and isotropic, then under some technical assumptions,
for all $f\in L^p(\Rd)$ we have
\begin{equation}\label{Hardy-intro}
\left(\int_{\Rd} |f(x)|^p L(|x|)\, \ud x \right)^{1/p} \leq c \left(\int_{\Rd} \int_{\Rd} |f(x+h)-f(x)|^p \,\ud x \, \nu (\ud h) \right)^{1/p}
\end{equation}
for some $c>0$, where $L(r) = \int_{B_r^c}\nu( \ud x)$, see Corollary \ref{cor-Hardy-2} and Theorem \ref{thm-4-6}. 
This result extends \cite{MR1624754} to the  multidimensional case.
To the best knowledge of the authors, 
this is the first result in the direction where 
$$
\lim_{x\to 0} \frac{L(|x|)|x|^{-d}}{\frac{\nu(\ud x)}{\ud x}} = \infty.
$$
In a typical form of nonlocal Hardy inequalities the limit above is equal to $1$, see e.g. Dyda and V\"ah\"akangas \cite{DydaVahakangas}. Hardy inequality \eqref{Hardy-intro} leads to the following Sobolev embedding (see Theorem \ref{thm-3-5}):
\begin{align*}
\int_{\Rd} |f(x)|^p h_p(|x|) \,\ud x \leq c
\|f\|_{\Wp},
\end{align*}
for some $c>0$, where $h_p(r) = \int_{\Rd} \left(1\wedge (|x|/r)^p \right) \nu (\ud x)$. We would like to emphasize that our proof is different from the one in \cite{MR4313960} where the authors use the properties of the logarithm to prove their claim (see also \cite{MR3995732}). Finally, we obtain the following estimate of the nonlocal capacity of a ball:
	\begin{equation*}
	\Cap_{\nu,p} (B(x,r)) \approx r^d \left(1+ h_p(r)\right), \quad r>0, \, x\in\Rd.
	\end{equation*}
see Theorem \ref{21}. Further studies of nonlocal capacities could lead to establishing a Wiener test for a new class of nonlocal operators, similar to \cite{MR4595613}. 

\section{Preliminaries}
By $a \lesssim b$ we mean that there exists $c>0$ such that $a \leq cb$, and $a \approx b$ means that $a\lesssim b$ and $a\gtrsim b$, that is, $b\lesssim a$.
\subsection{Lower Matuszewska index and regular variation} 
Assume that a function $\varphi \colon (0,\infty) \to (0,\infty)$ satisfies 
\begin{equation}\label{O-reg-var-1} 
	\frac{\varphi(r_2)}{\varphi(r_1)}
	\geq 	A \left(\frac{r_2}{r_1}\right)^{a}
\end{equation}
when $0<r_1<r_2<R_0$ for some constants $A, R_0>0$, $a\in\R$. 
The supremum $\alpha$ of the
numbers $a$ for which the inequality above is satisfied for some $A, R_0 > 0$ is called the {\it lower
Matuszewska index} of $f$ at zero.

The supremum $\alpha$ of the
numbers $a$ for which the inequality \eqref{O-reg-var-1} is satisfied when $R_0<r_1<r_2$
for some $A, R_0 > 0$ is called the lower Matuszewska index of $f$ at infinity.

Noteworthy, if $\varphi$ has lower Matuszewska index at zero, then the function $1/\varphi(1/r)$ has the same lower Matuszewska index at infinity.

We say that a function $\varphi \colon (0,\infty) \to [0,\infty)$ is {\it regularly varying} at zero with index $\alpha$ (see \cite[Section 1.4.2]{Bingham}) if
$$
\lim_{r \to 0^+} \frac{\varphi (\lambda r)}{\varphi(r)} = \lambda^{\alpha} \quad \text{for all} \,\, \lambda>0.
$$
We note that such $\varphi$ has lower Matuszewska index at zero equal to $\alpha$. 

\subsection{Concentration functions}
Let $p\in[1,\infty)$. For a Borel measure $\nu$ on $\Rd$ satisfying 
\begin{equation}\label{p-integrability-nu}
\int_{\Rd} \left(1 \wedge |y|^p\right) \nu(\ud y) <\infty
\end{equation}
we let
\begin{equation}\label{hp}
	h_p(r) = \int_{\Rd} 
	\left( 1 \wedge \frac{|x|^p}{r^p} \right) \nu (\ud x), \quad r>0.
\end{equation}
For $r>0$ we also let
$$
L(r) = \int_{B_r^c} \nu(\ud x).
$$

\begin{lemma}\label{lem:h-L}
	 If $L$ has lower Matuszewska index at zero strictly bigger than $-p$, then there exists $r_0>0$ such that
\begin{equation}\label{eq:hp-L}
	h_p(r) \approx L(r), \quad \quad r<r_0. 
\end{equation}
If, additionally, $L$ has lower Matuszewska index at infinity strictly bigger than $-p$, then \eqref{eq:hp-L} holds for all $r>0$. 
\end{lemma}

\begin{proof}
	We have
	\begin{align*}
		h_p(r) = \int_{\Rd} \left( 1\wedge \frac{|x|^p}{r^p}\right) \, \nu(\ud x)
		&=\int_{\Rd} \left(\int_0^{r} \frac{ps^{p-1}}{r^{p}} \mathds{1}_{B_s^c}(x) \, \ud s\right)\, \nu(\ud x) \\
		&=\frac{p}{r^{p}} \int_0^r s^{p-1}\int_{B_s^c} \, \nu (\ud x) \, \frac{\ud s}{s}
		=\frac{p}{r^p} \int_0^r s^p L(s) \frac{\ud s}{s} \approx L(r)
	\end{align*}
	for $r<r_0$ by \cite[(ix)]{MR466438}.
\end{proof}
If $\nu$ is not compactly supported, then \eqref{eq:hp-L} holds for any $r_0>0$. 

\subsection{Dense subspace and min-max property of $W^{\nu}_p$}
First we will state a certain density result. We omit the proof as it is analogous to the one from \cite{MR3310082}.
The following result was proved in a weaker form in \cite[Theorem 3.66]{Guy}, where the author made the additional assumption that a L\'evy measure has a symmetric density.
\begin{lemma}\label{lemma}
Let $p\in[1,\infty)$.
Then $C_c^{\infty}(\Rd)$ is dense in $W_p^{\nu}$, i.e., for any $f\in\Wp$ there exists a sequence $\{f_j\}_{j\in\mathbb{N}} \subset C_c^{\infty}(\Rd)$ such that $\lim_{j\to\infty} \|f_j-f\|_{\Wp} =0$. Moreover, if $K$ is a compact set in $\Rd$ satisfying $K\subset \operatorname{int} (\{x\in\Rd: f(x) \geq 1\})$, then the sequence $\{f_j\}_{j\in\mathbb{N}}$ can be chosen to satisfy $f_j(x)\geq 1$ for all $(x,j)\in K\times \mathbb{N}$.
\end{lemma}

\begin{lemma}\label{lemminmax}
If $f_1, f_2 \in W_p^{\nu}$, then the functions $g:=\max \{f_1,f_2\}$
and $h:= \min\{f_1,f_2\}$ belong to $W_p^{\nu}$ with
$$
\|g\|_{W_p^{\nu}}^p + \|h\|_{W_p^{\nu}}^p \leq \|f_1\|_{W_p^{\nu}}^p +\|f_2\|_{W_p^{\nu}}^p.
$$
\end{lemma}

\section{Properties of nonlocal capacity}
For any $a,b \in \R$, we have
$$
\left||a|-|b|\right| \leq \left|a-b\right|,
$$
and therefore
$$
\| |f|\|_{\Wp} \leq \|f\|_{\Wp}.
$$
Hence, for any set $E\subset \Rd$, we have
\begin{equation}
\Cap_{\nu,p} (E) = \inf \{ \|f\|_{W^{\nu}_p}: 0 \leq f \in W_p^{\nu} \,\, \textnormal{and} \,\, E \subset \operatorname{int}(\{f\geq 1\}) \}.\label{eq:cap_al}
\end{equation}
We will present some basic properties of the $\Wp$-capacity, in analogy to the classical Sobolev capacity theory \cite[Section 4.7]{MR1158660}. 

\begin{proposition}\label{prop}
	The following assertions hold:
	\begin{enumerate}
		\item[(i)] Zero-property: $\Cap_{\nu,p} (\emptyset) = 0$.
		\item[(ii)] Monotonicity: For any sets $E_1, E_2 \subset \Rd$ satisfying $E_1 \subset E_2$,
		$$
		\Cap_{\nu,p}(E_1) \leq \Cap_{\nu,p} (E_2).
		$$
		\item[(iii)] Strong subadditivity: For any sets $E_1, E_2 \subset \Rd$,
		$$
		\Cap_{\nu,p} (E_1 \cup E_2) + \Cap_{\nu,p} (E_1 \cap E_2) \leq \Cap_{\nu,p} (E_1) + \Cap_{\nu,p} (E_2).
		$$
		\item[(iv)] Upper semi-continuity: If $\{K_j\}_{j\in\NN}$ is a decreasing sequence of compact sets in $\Rd$, then
		$$
		\Cap_{\nu,p} \left(\,\bigcap_{j=1}^{\infty} K_j\right) = \lim_{j \to \infty} \Cap_{\nu,p}(K_j).
		$$
		\item[(v)] Lower semi-continuity: If $\{E_j\}_{j \in \NN}$ is an increasing sequence of sets in $\Rd$, then
		$$
		\Cap_{\nu,p} \left(\,\bigcup_{j=1}^{\infty}E_j\right) = \lim_{j\to\infty} \Cap_{\nu,p}(E_j).
		$$
		\item[(vi)] Countable subadditivity: For any sequence of sets $\{E_j\}_{j\in\NN}$ in $\Rd$,
		$$
		\Cap_{\nu,p}\left(\,\bigcup_{j=1}^{\infty} E_j\right) \leq \sum_{j=1}^{\infty} \Cap_{\nu,p} (E_j). 
		$$
	\end{enumerate}
\end{proposition}
The proof of the above properties is based on standard arguments, in the proof of $(\rm{iii})$ and $(\rm{v})$ we also employ Lemma \ref{lemminmax}.
By Proposition \ref{prop}
and \cite[Theorem 2.3.11]{MR1411441}, we obtain the following inner/outer regularity of $\Cap_{\nu, p}$.
\begin{corollary}\label{cor2}
	If $E\subset \Rd$ is a Borel set, then
	$$
	\Cap_{\nu,p}(E) = \sup_{\textnormal{compact}\, K\subset E} \Cap_{\nu,p}(K) = \inf_{\textnormal{open}\, G\supset E} \Cap_{\nu,p}(G).
	$$
\end{corollary}
\subsection{Equivalent definitions of $\Cap_{\nu,p}$}
In this subsection we formulate two other equivalent ways of defining $\Cap_{\nu,p}$.
\begin{definition}
	For any compact set $K \subset \Rd$, let
	\begin{equation}
	\Cap_{\nu,p}^{\dagger}(K) := \inf \{ \|f\|_{\Wp}^p: f \in C_c^{\infty}(\Rd) \,\, \textnormal{and} \,\, f \geq 1 \,\, \textnormal{on} \,\, K \}. \label{eq:capd}
	\end{equation}
\end{definition} 
Furthermore, we extend this notion from compact sets to general sets as follows:
\begin{itemize}
\item for any open set $G\subset \Rd$, we define
\begin{equation}
\Cap_{\nu,p}^{\dagger}(G):= \sup \{ 	\Cap_{\nu,p}^{\dagger}(K): \textnormal{compact} 
\, K \subset G\};\label{eq:capdk}
\end{equation}
\item for an arbitrary set $E\subset \Rd$, we define
\begin{equation}
\Cap_{\nu,p}^{\dagger}(E) := \inf \{	\Cap_{\nu,p}^{\dagger}(G): \textnormal{open}\,\, G \supset E\}. \label{eq:capdo}
\end{equation}
\end{itemize}
Note that for open or compact set $E$, \eqref{eq:capdo} agrees with \eqref{eq:capdk} and \eqref{eq:capd}, respectively.

A standard procedure along with Lemma \ref{lemma} and Corollary \ref{cor2} give us the following result.

\begin{proposition}\label{prop2}
	If $E\subset \Rd$, then 
	$$
	\Cap^{\dagger}_{\nu,p}(E)=\Cap_{\nu,p}(E)
	$$
\end{proposition}

Following Netrusov \cite[Definition 2]{MR1172762}, for an arbitrary set $E\subset\Rd$, we let
\begin{align*}
	\Cap_{\nu,p}^{\dagger \dagger} (E) &:= \inf_{\textnormal{open set} \, G \supset E} 	\Cap_{\nu,p}^{\dagger \dagger} (G)\\
	&:=\inf_{\textnormal{open set} \, G \supset E} \inf \{\|f\|^p_{\Wp} : f \in \Wp \,\, \textnormal{and} \,\,  f \geq 1 \, \textnormal{a.e. on} \,\, G\}.
\end{align*}

\begin{proposition}\label{prop-2-4}
	For any set $E\subset\Rd$,
	$$
	\Cap_{\nu,p}^{\dagger} (E) = 	\Cap_{\nu,p}^{\dagger\dagger} (E)  =	\Cap_{\nu,p}(E).
	$$
\end{proposition}
The proof is based on standard arguments together with 
Lemma \ref{lemma}, Corollary \ref{cor2}, and Proposition \ref{prop2}.

\subsection{Capacitary inequality}
The next result is the capacitary inequality for $\Wp$.
\begin{proposition}\label{210}
For all lower semi-continuous functions $f\in\Wp$,
$$
\int_0^{\infty} \Cap_{\nu,p} \{ (|f|>t)\} \, \ud t^p \lesssim \|f\|_{\Wp}^p.
$$
\end{proposition}
With Proposition \ref{prop} at hand, the proof is analogous to the one of \cite[Theorem 1]{MR1637452}.

Denote by $\mathcal{M}$ the uncentered Hardy--Littlewood maximal operator on $\Rd$, that is, for any locally integrable function $f$ on $\Rd$ and any $x\in\Rd$,
$$
\mathcal{M} f(x):= \sup_{\{B\ni x: B \, \textnormal{is a ball of} \, \Rd\}} \frac{1}{|B|} \int_{B} |f(z)| \, \ud z.
$$
As a corollary of Proposition \ref{210}, we can show the following strong capacitary inequality for $\mathcal{M}$ similarly as \cite[Lemma 3.1]{MR1637452}.
\begin{corollary}
For all $f\in\Wp$, 
$$
\int_0^{\infty} \Cap_{\nu,p} (\{|\mathcal{M}f|>t\}) \,\ud t^p \lesssim \|f\|_{\Wp}^p.
$$
\end{corollary}

\subsection{Connection between $\Cap_{\nu,p}$ and $\nu$-perimeter}

We recall that the $\nu$-perimeter of a set $E\subset\Rd$ is given by
$$
\Per_{\nu} (E) : = \int_{E}\int_{E^c-x} \nu (\ud y) \ud x,
$$
see \cite{Cygan-Grzywny-Per}. The following result is an application of \cite[Proposition 2.7]{Cygan-Grzywny-Per} and shows that $\Per_{\nu}$ can be used to derive a co-area formula for $\| \cdot \|_{W_1^{\nu}}$.

\begin{lemma}\label{lem31}
Let $0\leq f \in W_{1}^{\nu}$. Set $S_t(f) = \{x\in\Rd: f(x)>t\}$, $t>0$. Then
$$
\|f\|_{W_{1}^{\nu}} = \int_0^{\infty} \left(2\Per_{\nu} (S_t(f)) + |S_t(f)| \right) \ud t.
$$
\end{lemma}
With Proposition \ref{prop-2-4} and Lemma \ref{lem31} at hand, we can obtain the following geometric characterization of $\Cap_{\nu,1}$. 
\begin{theorem}
Let $K$ be a compact subset of $\Rd$. Then
$$
\Cap_{\nu,1} (K) = \Cap^{\star}_{\nu,1} (K),
$$
where
$$
\Cap^{\star}_{\nu,1} (K) := \inf \{\Per_{\nu}(O) + |O|: \textnormal{open} \, O \supset K \, \textnormal{with compact} \, \overline{O}\}.
$$
\end{theorem}

\section{Hardy inequality, Sobolev embedding and ball's capacity}
In this section, we assume that the measure $\nu$ is absolutely continuous with a symmetric density. At first, we obtain a series of Hardy inequalities. Then, we apply one of them to prove the nonlocal Sobolev embedding. Finally, we estimate the nonlocal capacity of a ball.

\subsection{Hardy inequalities and nonlocal Sobolev embedding}
The following theorem extends \cite[Theorem 2.1]{MR1624754} to the multi-dimensional case. 

\begin{lemma}\label{thm-Hardy-cap-1}
Let $p\in[1,\infty)$, $0<A\leq \infty$ and let $\nu$ be a non-negative symmetric (i.e. $\nu(x)=\nu(-x)$) 
function on $\{x=(x_1, ..., x_d) \in \Rd: |x_d|<2A\}$ such that
$$
w(s) \coloneqq \int_{\{s<y_d<2A\}} \nu(y)\, \ud y <\infty
$$
for all $s\in(0,2A)$. Suppose that there exists $\beta\in(1,2)$ such that 
\begin{equation}\label{eq:v-scal}
w(s) \leq \beta w(2s), \quad s\in(0,A/2).
\end{equation}
Then, for all $a\in(0,A]$ and all $f\in L^p(\{x: 0<x_d<a\})$ we have
\begin{equation}\label{Hardy-a-A}
\begin{split}
\left(\int_{\{0<x_d<a\}} |f(x)|^p w(x_d) \, \ud x\right)^{1/p} &\leq C_1  \Bigg[
w(a)^{1/p} \left(\int_{\{0<x_d<a\}} |f(x)|^p \,\ud x\right)^{1/p}\\
&\quad + \left(\int_{\{0<y_d<a\}} \int_{\{0<x_d<a\}} |f(x)-f(y)|^p \nu(x-y) \,\ud x \ud y\right)^{1/p} \Bigg],
\end{split}
\end{equation}
where $C_1$ is independent of $f$ and $a$. 

In particular, if $a=A=\infty$, then for all $f\in L^p(\Rd_+)$, where $\Rd_+ = \{x=(x_1, \ldots, x_d) \in \Rd: x_d>0\}$, we have
\begin{equation}\label{eq:ineq_v_xd}
\left(\int_{\Rd_+} |f(x)|^p w(x_d)\, \ud x \right)^{1/p} \leq C_1 \left(\int_{\Rd_+} \int_{\Rd_+} |f(x)-f(y)|^p \nu(x-y)\, \ud x \ud y\right)^{1/p}.
\end{equation}
\end{lemma}

\begin{proof}
Let
$$
V(t) \coloneqq \inf \{s\geq 0: w(s)\leq t\}, \quad t\geq0. 
$$
This function is nonnegative, nonincreasing, right continuous and we have
\begin{equation}\label{eq:vV}
V[w(s)] \leq s, \quad w[V(s)] = s.
\end{equation}
Note that by \eqref{eq:v-scal} and \eqref{eq:vV},
\begin{equation}\label{eq:delta}
\delta(s)  \coloneqq V[\beta w(s)] \leq V[w(s/2)] \leq s/2, \quad 0<s<A.
\end{equation}
Since $w[\delta(s)] = \beta w(s)$, we have
\begin{equation}\label{eq:v-delta}
\int_{\{\delta(s)<y_d<s\}} \nu(y) \,\ud y = (\beta-1) w(s).
\end{equation}
Let us observe that if $x_d= y_d - \delta(y_d)$, then by \eqref{eq:delta} $x_d<y_d<2x_d$, 
hence for any $\eps \in (0,a/2)$, by Minkowski inequality
\begin{align*}
&\left(\int_{\eps<x_d<a/2} \int_{2x_d<y_d<a} |f(x)|^p \nu(y-x) \, \ud y \ud x\right)^{1/p} \\
&\qquad\qquad\quad\leq \left(\int_{2\eps<y_d<a} \int_{0<x_d<y_d-\delta(y_d)} |f(y)|^p \nu(y-x) \,\ud x \ud y\right)^{1/p} \\
&\qquad\qquad\quad\quad+
\left(\int_{0<y_d<a} \int_{0<x_d<y_d} |f(x)-f(y)|^p \nu(y-x) \,\ud x \ud y\right)^{1/p} 
\end{align*}
and hence
$$
\left(\int_{\eps<x_d<a/2} |f(x)|^p \int_{x_d<z_d<a-x_d} \nu(z) \,\ud z \ud x \right)^{1/p}
 \leq \left(\int_{\eps<y_d<a} |f(y)|^p \int_{\delta(y_d)<z_d<y_d} \nu(z) \, \ud z \ud y \right)^{1/p} + \Lambda,
$$
where 
$$
\Lambda = \frac{1}{2} \left(\int_{0<y_d<a} \int_{0<x_d<a} |f(x)-f(y)|^p \nu(y-x)\, \ud x \ud y\right)^{1/p}.
$$
Hence, by \eqref{eq:v-delta}
\begin{align*}
\left(\int_{\eps<x_d<a/2} |f(x)|^p [w(x_d)-w(a-x_d)]\, \ud x \right)^{1/p}
&\leq (\beta-1)^{1/p} \left(\int_{\eps<x_d<a} |f(x)|^p w(x_d)\, \ud x\right)^{1/p} + \Lambda\\
&\leq (\beta-1)^{1/p} \left(\int_{\eps<x_d<a/2} |f(x)|^p w(x_d) \,\ud x\right)^{1/p}\\
&\quad+ [\beta(\beta-1) w(a)]^{1/p} \left(\int_{a/2<x_d<a} |f(x)|^p \,\ud x\right)^{1/p}  + \Lambda
\end{align*}
since for $x_d \in (a/2, a)$ we have $w(x_d) \leq w(a/2) \leq \beta w(a)$ by \eqref{eq:v-scal}. Also $w(a-x_d)\leq w(a/2) \leq \beta w(a)$ for $x_d \in [0,a/2]$ and thus
\begin{align*}
\left(\int_{\eps<x_d<a/2} |f(x)|^p w(x_d)\, \ud x\right)^{1/p} 
&\leq
\left(\int_{\eps<x_d<a/2} |f(x)|^p [w(x_d)-w(a-x_d)]\, \ud x\right)^{1/p}  \\
&\quad + \left(\int_{\eps<x_d<a/2} |f(x)|^p w(a-x_d) \,\ud x\right)^{1/p} \\
&\leq (\beta-1)^{1/p} \left( \int_{\eps<x_d<a/2} |f(x)|^p w(x_d) \,\ud x \right)^{1/p} \\
&\quad +[\beta(\beta-1) w(a)]^{1/p} \left(\int_{a/2<x_d<a} |f(x)|^p \, \ud x\right)^{1/p}
+\Lambda\\
&\quad +\beta^{1/p} w(a)^{1/p} \left(\int_{0<x_d<a/2} |f(x)|^p\, \ud x \right)^{1/p}.
\end{align*}
Therefore
\begin{align*}
\left(\int_{\eps <x_d<a/2} |f(x)|^p w(x_d) \,\ud x\right)^{1/p}
&\leq
\frac{1}{1-(\beta-1)^{1/p}} \left(\beta^{1/p}w(a)^{1/p} \left(\int_{0<x_d<a} |f(x)|^p \,\ud x\right)^{1/p} + \Lambda \right).
\end{align*}
Finally,
\begin{align*}
\left(\int_{\eps<x_d<a} |f(x)|^p w(x_d)\, \ud x\right)^{1/p}
&\leq
\left(\int_{\eps<x_d<a/2} |f(x)|^p w(x_d)\, \ud x\right)^{1/p}\\
&\quad 
+[\beta w(a)]^{1/p} 
\left(\int_{a/2<x_d<a}|f(x)|^p\, \ud x\right)^{1/p}	\\
&\leq \left(1+ \frac{1}{1-(\beta-1)^{1/p}} \right) [\beta w(a)]^{1/p} \left(\int_{0<x_d<a} |f(x)|^p\, \ud x \right)^{1/p} \\
&\quad + \frac{1}{1-(\beta-1)^{1/p}} \Lambda.
\end{align*}
Letting $\eps\to0$ we obtain
\begin{align*}
\left(\int_{0<x_d<a} |f(x)|^p w(x_d)\, \ud x \right)^{1/p}
&\leq C_1 \Bigg(w(a)^{1/p} \left(\int_{0<x_d<a} |f(x)|^p\, \ud x\right)^{1/p}\\
&\quad+ \left(\int_{0<y_d<a} \int_{0<x_d<a}|f(x)-f(y)|^p \nu(x-y)\, \ud x \ud y \right)^{1/p} \Bigg),
\end{align*}
where $C_1 = \left(1+ \frac{1}{1-(\beta-1)^{1/p}} \right) \beta^{1/p}$. If $A=\infty$, by taking $a\to\infty$ we get
\begin{align*}
\left(\int_{\Rd_+} |f(x)|^p w(x_d)\, \ud x\right)^{1/p}
\leq
C_1 \left(\int_{\Rd_+} \int_{\Rd_+} |f(x)-f(y)|^p \nu(x-y)\, \ud x \ud y\right)^{1/p}.
\end{align*}
\end{proof}

Using a change of variables, monotonicity of $w$ and 
\eqref{Hardy-a-A}
applied to $f(-\cdot)$, we obtain the following result.

\begin{corollary}
Let $p\in[1,\infty)$ and let $\nu$ be a non-negative symmetric function on $\Rd$ such that
$$
w(s) \coloneqq \int_{\{y_d>s\}} \nu(y) \,\ud y <\infty
$$
for all $s\in(0,\infty)$. Suppose that there exists $\beta\in(1,2)$ such that 
\begin{equation}
\label{eq:v-scal-32}
w(s) \leq \beta w(2s), \quad s\in(0,\infty).
\end{equation}
Then for all $a\in(0,\infty]$ and $f\in L^p(\{x: |x_d|<a\})$ we have
\begin{equation}\label{Hardy-a-A-2}
\begin{split}
\left(\int_{\{|x_d|<a\}} |f(x)|^p w(|x|) \, \ud x\right)^{1/p} &\leq 2C_1 \Bigg[
w(a)^{1/p} \left(\int_{\{|x_d|<a\}} |f(x)|^p \,\ud x\right)^{1/p}\\
&\quad + \left(\int_{\{|y_d|<a\}} \int_{\{|x_d|<a\}} |f(x)-f(y)|^p \nu(x-y) \,\ud x \ud y\right)^{1/p} \Bigg],
\end{split}
\end{equation}
In particular, for all $f\in L^p(\Rd)$,
\begin{equation}\label{eq:ineq-Rd}
\left(\int_{\Rd} |f(x)|^p w(|x|)\, \ud x \right)^{1/p} \leq 2C_1 \left(\int_{\Rd} \int_{\Rd} |f(x)-f(y)|^p \nu(x-y)\, \ud x \ud y\right)^{1/p}.
\end{equation}
\end{corollary}

Under additional assumption that $\nu$ is radially symmetric, we obtain
a more useful version of the corollary above.

\begin{corollary}\label{cor-Hardy-2}
	Let $p\in[1,\infty)$ and let $\nu$ be a non-negative radially symmetric function on $\Rd$ such that
	\begin{equation}\label{w-finite-Hardy-1}
	w(s) \coloneqq \int_{\{y_d>s\}} \nu(y)\, \ud y <\infty
	\end{equation}
	for all $s\in(0,\infty)$. Suppose that there exists $\beta\in(1,2)$ such that
	\begin{equation}\label{eq:vi-scal2}
		w(s) \leq \beta w(2s), \quad s\in(0,\infty).
	\end{equation}
	Then for all $a\in(0,\infty]$ and $f\in L^p(\{x: |x_d|<a\})$ we have
\begin{equation}\label{Hardy-a-A-5}
\begin{split}
\left(\int_{\{|x_d|<a\}} |f(x)|^p L(|x|) \, \ud x\right)^{1/p} &\lesssim \Bigg[
w(a)^{1/p} \left(\int_{\{|x_d|<a\}} |f(x)|^p \,\ud x\right)^{1/p}\\
&\quad + \left(\int_{\{|y_d|<a\}} \int_{\{|x_d|<a\}} |f(x)-f(y)|^p \nu(x-y) \,\ud x \ud y\right)^{1/p} \Bigg],
\end{split}
\end{equation}
In particular, for all $f\in L^p(\Rd)$ we have
	$$
	\left(\int_{\Rd} |f(x)|^p L(|x|)\, \ud x \right)^{1/p} \lesssim \left(\int_{\Rd} \int_{\Rd} |f(x)-f(y)|^p \nu(x-y) \,\ud x \ud y\right)^{1/p}.
	$$	
\end{corollary}

\begin{proof}
For $i=1,\ldots,d$ let
$$
w_i(s) \coloneqq \int_{\{y_i>s\}} \nu(y)\, \ud y, \quad s\in(0,\infty).
$$
Note that since $\nu$ is radially symmetric, $w(s) = w_i(s)$ for all $i\in\{1,\ldots, d\}$ and $s\in(0,\infty)$. 
Observe that an analogue of \eqref{Hardy-a-A-2} holds for all $i\in\{1,\ldots, d\}$, $a\in(0,\infty]$ and $f\in L^p(\{|x_i|<a\})$:
\begin{equation}\label{eq:ineqs1}
\begin{split}
\left(\int_{\{|x_i|<a\}} |f(x)|^p w_i(|x|) \, \ud x\right)^{1/p} &\leq c_1 \Bigg[
w_i(a)^{1/p} \left(\int_{\{|x_i|<a\}} |f(x)|^p \,\ud x\right)^{1/p}\\
&\quad + \left(\int_{\{|y_i|<a\}} \int_{\{|x_i|<a\}} |f(x)-f(y)|^p \nu(x-y) \,\ud x \ud y\right)^{1/p} \Bigg],
\end{split}
\end{equation}
Further, $|y|>s$ implies $|y_i|>s/\sqrt{d}$ for some $i\in\{1, \ldots, d\}$ and hence
\begin{align*}
\begin{split}
L(s) = \int_{\{|y|>s\}} \nu(y) \,\ud y 
&\leq \int_{\cup_i \{|y_i|>s/\sqrt{d}\}} \nu(y)\, \ud y  
\leq \sum_i \int_{ \{|y_i|>s/\sqrt{d}\}} \nu(y)\, \ud y \\
& = 2 \sum_i \int_{ \{y_i>s/\sqrt{d}\}} \nu(y) \,\ud y = 2\sum_i w_i\left(\frac{s}{\sqrt{d}}\right) = 2d  w\left(\frac{s}{\sqrt{d}}\right).
\end{split}
\end{align*}
Moreover, by \eqref{w-finite-Hardy-1},
\begin{equation}\label{vi-alfa-N-estim}
w\left(\frac{s}{\sqrt{d}}\right) \leq \beta^N w \left(\frac{2^N}{\sqrt{d}} \,s\right) \leq \beta^N w (s),
\end{equation}
where $N := \inf \{n\in \mathbb{N}: 2^n/\sqrt{d} \geq 1\} = \left\lceil \log_2 \sqrt{d} \right\rceil$ (note that $\beta^N \leq \beta \sqrt{d}  \leq 2\sqrt{d}$). Hence
$$
L(s) \leq 2d\beta^N w(s).
$$
Therefore, by \eqref{Hardy-a-A-2}
\begin{align*}
\begin{split}
\left(\int_{\{|x_d|<a\}} |f(x)|^p L(|x|) \, \ud x\right)^{1/p} &\leq 2^{1+1/p}C_1 (d\beta^N)^{1/p} \Bigg[
w(a)^{1/p} \left(\int_{\{|x_d|<a\}} |f(x)|^p \,\ud x\right)^{1/p}\\
&\quad + \left(\int_{\{|y_d|<a\}} \int_{\{|x_d|<a\}} |f(x)-f(y)|^p \nu(x-y) \,\ud x \ud y\right)^{1/p} \Bigg].
\end{split}
\end{align*}
\end{proof}

We can now prove the embedding result. For related nonlocal Sobolev embeddings, see \cite[Theorem 5.8 (i))]{foghem2023stability}.
\begin{theorem}\label{thm-3-5}
Assume that $\nu$ is a non-negative, radially symmetric function and 
has lower Matuszewska index at zero 
strictly bigger than $-d-1$.
Then 
\begin{align*}
\int_{\Rd} |f(x)|^p h_p(|x|) \,\ud x \lesssim 
\|f\|_{\Wp}.
\end{align*}
\end{theorem}
\begin{proof}
It is easy to show that $L$ has lower Matuszewska index at zero strictly bigger than $-1$, hence by Lemma \ref{lem:h-L}, there exists $r_0>0$ such that $h_p \approx L$ for $r<r_0$. 

There exist $a<d+1$, $A \leq 1$, $R>0$ such that
\begin{equation}\label{nu-scal-11}
\frac{\nu(r_2)}{\nu(r_1)} \geq A \left(\frac{r_2}{r_1}\right)^{-a}, \quad 0<r_1\leq r_2<R.
\end{equation}
Let $g(r)\coloneq r^a \nu(r)$, $r\in(0,R)$. By \eqref{nu-scal-11}, $g(r_2) \geq A g(r_1)$ for $0<r_1\leq r_2<R$. We let $\tilde{g}(r) = \sup_{s\in(0,r]} g(s) = \sup_{\lambda \in(0,1]} g(\lambda r)$. Of course, $\tilde{g}$ is non-increasing. Moreover, we have $ g(r) \leq \tilde{g} (r) \leq A^{-1}g(r)$. 
Define $\tilde{\nu}(r) = r^{-a} \tilde{g}(r)$, $r\in(0,R)$. It is clear that $\tilde{\nu} \approx \nu$ on $(0,R)$. Furthermore,
$$
\tilde{\nu} (r_2) = (r_2)^{-a} \tilde{g} (r_2) \geq \left(\frac{r_2}{r_1}\right)^{-a} r_1^{-a} \tilde{g} (r_1) = \left(\frac{r_2}{r_1}\right)^{-a} \tilde{\nu}_2(r_1), \qquad 0<r_1\leq r_2<R.  
$$
Let $\tilde{\nu}(r) = \tilde{\nu}(R) R^a r^{-(a \vee (d+1/2))}$ for $r>R$. 
Observe that 
$$
\tilde{\nu}(2r) \geq  2^{-(a \vee (d+1/2))} \tilde{\nu}(r), \quad r \in (0,\infty).  
$$ 
We will verify that $\tilde{\nu}$ satisfies the assumptions of Corollary \ref{cor-Hardy-2}. Indeed, we have
\begin{align*}
\tilde{w}(2s) = \int_{\{y_d>2s\}} \tilde{\nu}(y) \,\ud y &= 2^d \int_{\{y_d>s\}} \tilde{\nu}(2y)\, \ud y\\
& \geq 2^{d-(a \vee (d+1/2))} \int_{\{y_d>s\}} \tilde{\nu}(y)\, \ud y = 2^{d-(a \vee (d+1/2))} \tilde{w}(s), \quad 
\end{align*}
where $(a \vee (d+1/2))-d<1$. Let $\delta \coloneq R \wedge r_0$. 
By Corollary \ref{cor-Hardy-2}, for all $f\in L^p(\{x:|x_d|<\delta\})$
\begin{align*}
\left(\int_{\{|x_d|<\delta\}} |f(x)|^p \tilde{L}(|x|) \, \ud x\right)^{1/p} &\lesssim  \Bigg[
\tilde{w}(\delta)^{1/p} \left(\int_{\{|x_d|<\delta\}} |f(x)|^p \,\ud x\right)^{1/p}\\
&\quad + \left(\int_{\{|y_d|<\delta\}} \int_{\{|x_d|<\delta\}} |f(x)-f(y)|^p \tilde{\nu}(x-y) \,\ud x \ud y\right)^{1/p} \Bigg].
\end{align*}
Observe that
\begin{align*}
\int_{\{|x|\geq \delta\}} |f(x)|^p \int_{\Rd} 
 \left( 1 \wedge \frac{|h|^p}{|x|^p}\right)
 \nu (h) \,\ud h\, \ud x
&\lesssim \int_{\{|x|\geq \delta\}} |f(x)|^p \int_{\Rd}
\left(1\wedge|h|^p\right) 
  \nu(h) \, \ud h \, \ud x\\
&\lesssim \int_{\{|x|\geq \delta\}} |f(x)|^p \,\ud x \lesssim \|f\|_p^p.
\end{align*}
Further, 
\begin{align*}
\int_{\{|x|<\delta\}} |f(x)|^p h_p(|x|) \,\ud x 
&\lesssim \int_{\{|x|<\delta\}} |f(x)|^p L(|x|) \,\ud x \\
&\lesssim \int_{\{|x_d|<\delta\}} |f(x)|^p \tilde{L}(|x|)\, \ud x \\
&\lesssim 
\tilde{w}(\delta)^{1/p} \left(\int_{\{|x_d|<\delta\}} |f(x)|^p \,\ud x\right)^{1/p}\\
&\quad + \left(\int_{\{|y_d|<\delta\}} \int_{\{|x_d|<\delta\}} |f(x)-f(y)|^p \tilde{\nu}(x-y) \,\ud x \ud y\right)^{1/p} \\
&\lesssim
w(\delta)^{1/p} \left(\int_{\{|x_d|<\delta\}} |f(x)|^p \,\ud x\right)^{1/p}\\
&\quad + \left(\int_{\{|y_d|<\delta\}} \int_{\{|x_d|<\delta\}} |f(x)-f(y)|^p \nu(x-y) \,\ud x \ud y\right)^{1/p}.
\end{align*}
\end{proof}

Under the additional assumption on the behavior of $\nu$ at infinity, we obtain the following Hardy-type inequality.

\begin{theorem}\label{thm-4-6}
Assume that $\nu$ is a non-negative, radially symmetric function and 
has lower Matuszewska index at zero 
strictly bigger than $-d-1$ 
and lower Matuszewska index at infinity strictly bigger than $-d-1$.
	Then for all $f\in L^p(\Rd)$ we have
	$$
	\left(\int_{\Rd} |f(x)|^p L(|x|)\, \ud x \right)^{1/p} \lesssim \left(\int_{\Rd} \int_{\Rd} |f(x)-f(y)|^p \nu(x-y) \,\ud x \ud y\right)^{1/p}.
	$$	
\end{theorem}
\begin{proof}
There exist $a<d+1$, $A\leq 1$, $R_1>0$ such that
\begin{equation}\label{nu-scal}
\frac{\nu(r_2)}{\nu(r_1)} \geq A \left(\frac{r_2}{r_1}\right)^{-a}, \quad R_1<r_1\leq r_2.
\end{equation}
Let $g_1(r)\coloneq r^a \nu(r)$, $r>R_1$. By \eqref{nu-scal}, $g_1(r_2) \geq A g_1(r_1)$ for $R_1<r_1\leq r_2$. We let $\tilde{g}_1(r) = \inf_{s\geq r} g_1(s) = \inf_{\lambda \geq 1} g_1(\lambda r)$. Of course, $\tilde{g}_1$ is non-increasing. Moreover, we have $A g_1(r) \leq \tilde{g}_1 (r) \leq g_1(r)$. 
Define $\tilde{\nu}_1(r) = r^{-a} \tilde{g}_1(r)$, $r>R_1$. It is clear that $\tilde{\nu}_1 \approx \nu$ on $(R_1,\infty)$. Furthermore,
$$
\tilde{\nu}_1(r_2) = (r_2)^{-a} \tilde{g}_1(r_2) \geq \left(\frac{r_2}{r_1}\right)^{-a} r_1^{-a} \tilde{g}_1(r_1) =  \left(\frac{r_2}{r_1}\right)^{-a} \tilde{\nu}_1(r_1), \qquad R_1<r_1\leq r_2.  
$$ 
Further,
there exist $b<d+1$, $B \leq 1$, $R_2>0$ such that
\begin{equation}\label{nu-scal-2}
\frac{\nu(r_2)}{\nu(r_1)} \geq B \left(\frac{r_2}{r_1}\right)^{-b}, \quad 0<r_1\leq r_2<R_2.
\end{equation}
Let $g_2(r)\coloneq r^b \nu(r)$, $r\in(0,R_2)$. By \eqref{nu-scal-2}, $g_2(r_2) \geq B g_2(r_1)$ for $0<r_1\leq r_2<R_2$. We let $\tilde{g}_2(r) = \sup_{s\in(0,r]} g_2(s) = \sup_{\lambda \in(0,1]} g_2(\lambda r)$. Of course, $\tilde{g}_2$ is non-increasing. Moreover, we have $ g_2(r) \leq \tilde{g}_2 (r) \leq B^{-1}g_2(r)$. 
Define $\tilde{\nu}_2(r) = r^{-b} \tilde{g}_2(r)$, $r\in(0,R_2)$. It is clear that $\tilde{\nu}_2 \approx \nu$ on $(0,R_2)$. Furthermore,
$$
\tilde{\nu}_2(r_2) = (r_2)^{-b} \tilde{g}_2(r_2) \geq \left(\frac{r_2}{r_1}\right)^{-b} r_1^{-b} \tilde{g}_2(r_1) = \left(\frac{r_2}{r_1}\right)^{-b} \tilde{\nu}_2(r_1), \qquad 0<r_1\leq r_2<R_2.  
$$ 
Without loss of generality, we may assume that $R_2/2<R_1$. 
Let 
$$
\tilde{\nu}(r) =
\begin{cases}
\tilde{\nu}_2(r),  &\quad r \in (0,R_2/2), \\
\tilde{\nu}_2(R_2/2), &\quad r \in [R_2/2, R_1],\\
\frac{\tilde{\nu}_2(R_2/2)}{\tilde{\nu}_1(R_1)}
\tilde{\nu}_1(r), &\quad r \in (R_1,\infty).
\end{cases}
$$
Observe that 
$$
\tilde{\nu}(2r) \geq  2^{-(a\vee b)} \tilde{\nu}(r), \quad r \in (0,\infty).  
$$ 
We will verify that $\tilde{\nu}$ satisfies the assumptions of Corollary \ref{cor-Hardy-2}. Indeed, we have
\begin{align*}
\tilde{w}(2s) = \int_{y_d>2s} \tilde{\nu}(y) \,\ud y &= 2^d \int_{\{y_d>s\}} \tilde{\nu}(2y)\, \ud y\\
& \geq 2^{d-(a\vee b)} \int_{\{y_d>s\}} \tilde{\nu}(y)\, \ud y = 2^{d-(a\vee b)} \tilde{w}(s), \quad 
\end{align*}
where $(a\vee b)-d<1$. 
\end{proof}
		
		We will illustrate Theorem \ref{thm-4-6} with the following two examples.
		\begin{example}\label{ex-alpha-01}
		Let $\nu(x) = |x|^{-d} \log^{\gamma}(1+|x|^{-\delta})$, where $\delta,\gamma>0$ and $\delta\gamma<1$. First note that for $r<1/2$, $\log(r^{-\delta}) \approx \log(1+r^{-\delta})$. Hence, for $s<1/4$,
		\begin{align*}
			\int_s^{1/2} \frac{1}{r} \log^{\gamma}\left(1+\frac{1}{r^{\delta}}\right) \ud r 
			\approx \int_s^{1/2} \frac{1}{r} \log^{\gamma}\left(\frac{1}{r}\right) \ud r 
			&= \frac{1}{\gamma+1} \left(\log^{\gamma+1} \left(\frac{1}{s}\right) - \log^{\gamma+1} (2) \right)\\
			&\approx
 \log^{\gamma+1} \left(1+\frac{1}{s^{\delta}}\right).
		\end{align*}
		Since $\int_{1/2}^{\infty} r^{-1} \log^{\gamma}(1+r^{-\delta})\, \ud r< \infty$, then $L(s) \approx \log^{\gamma+1} (1+{s^{-\delta}})$ for $s<1/4$. 
		Since $\log^{\gamma}(1+{s^{-\delta}}) \sim s^{-\delta\gamma}$ as $s\to\infty$, for $s\geq1/4$ we have
			\begin{align*}
			L(s) = |\Ss^{d-1}| \int_s^{\infty} r^{-1} \log^{\gamma}\left(1+ \frac{1}{r^{\delta}}\right) \ud r 
			\approx \int_s^{\infty} r^{-1-\delta\gamma} \ud r 
			= \frac{1}{\delta\gamma} s^{-\delta\gamma}.
		\end{align*}
		Next, we will estimate $L$ from below. 
			First note that
		 for $c<1$,
		\begin{equation}\label{ineq-c1-log}
		\log(1+cr) \geq c \log(1+r) ,\ \quad r\geq0.
		\end{equation}
		Therefore
		\begin{align*}
			L(s) = |\Ss^{d-1}| \int_s^{\infty} r^{-1} \log^{\gamma}\left(1+ \frac{1}{r^{\delta}}\right) \ud r 
			&\geq |\Ss^{d-1}|s^{\delta\gamma}\log^{\gamma}\left(1+ \frac{1}{s^{\delta}}\right)  \int_s^{\infty} r^{-1-\delta\gamma} \, \ud r \\
			&=\frac{|\Ss^{d-1}|}{\delta\gamma}
			\log^{\gamma}\left(1+ \frac{1}{s^{\delta}}\right) .
		\end{align*}
		Hence $L(s) \approx 
		\log^{\gamma}(1+ {s^{-\delta}})$ for $s\geq1/4$.
		
		It follows from 
		Theorem \ref{thm-4-6}
		that for $f\in L^p(\Rd)$, 
		\begin{align*}
		&\left(\int_{\Rd} |f(x)|^p 
		\log^{\gamma}
		\left(1+ \frac{1}{|x|^{\delta}}\right)
		\left(\log \left(1+ \frac{1}{|x|^{\delta}}\right)+ 1 \right)\ud x
		 \right)^{1/p} \\
		&\quad\quad\quad\quad\quad\quad
		\quad\quad\quad\quad\quad\quad \quad\quad\quad
		\lesssim \left(\int_{\Rd} \int_{\Rd} \frac{|f(x)-f(y)|^p} {|x-y|^{d}} \log^{\gamma}\left(1+\frac{1}{|x-y|^{\delta}}\right)\ud x \ud y\right)^{1/p}.
		\end{align*}
		Note that this inequality does not follow from \cite{DydaVahakangas}, since we have different powers of $\log(1+|x|^{-\delta})$ on both sides of the inequality.
	\end{example}	

\begin{example}
Let $\nu(x) = |x|^{-d} \log^{\beta}(2+|x|^{-1}) \log^{-\gamma}(2+|x|)$, where $\beta \geq -1$ and $\gamma>1$. Since the logarithm is slowly varying, $\nu$ has lower Matuszewska index $-d$ both at zero and at infinity, hence it satisfies the assumptions of Theorem \ref{thm-4-6}. 
For $\beta>-1$ we have
$$
L(s) \approx \log^{\beta+1} \left(2+r^{-1}\right) \log^{-\gamma+1} \left(2+r\right),
$$
hence for all $f\in L^p(\Rd)$
	$$
	\left(\int_{\Rd} |f(x)|^p \,
	 \frac{\log^{\beta+1} \left(2+|x|^{-1}\right)}{\log^{\gamma-1} \left(2+|x|\right)}
	\, \ud x \right)^{1/p} \lesssim \left(\int_{\Rd} \int_{\Rd} \frac{|f(x)-f(y)|^p}{|x-y|^{d}} \,\frac{\log^{\beta}(2+|x-y|^{-1})}{\log^{\gamma}(2+|x-y|)} \,\ud x \ud y\right)^{1/p}.
	$$	
	If $\beta=-1$, then
$$
L(s) \approx \log\left(\log \left(2+r^{-1}\right)\right) \log^{-\gamma+1} \left(2+r\right),
$$
hence for all $f\in L^p(\Rd)$
	$$
	\left(\int_{\Rd} |f(x)|^p \,
	 \frac{\log\left(\log \left(2+|x|^{-1}\right)\right)}{\log^{\gamma-1} \left(2+|x|\right)}
	\, \ud x \right)^{1/p} \lesssim \left(\int_{\Rd} \int_{\Rd} \frac{|f(x)-f(y)|^p}{|x-y|^{d}} \,\frac{\log^{-1}(2+|x-y|^{-1})}{\log^{\gamma}(2+|x-y|)} \,\ud x \ud y\right)^{1/p}.
	$$	
\end{example}

\subsection{Ball's capacity}
We will use the results from previous subsection to obtain estimates of the nonlocal capacity of a ball.

\begin{theorem}\label{21}
Assume that $\nu$ is absolutely continuous with a radially symmetric density whose radial profile 
has lower Matuszewska index at zero 
strictly bigger than $-d-1$.
Then the following holds:
	\begin{equation}\label{eq:21}
		\Cap_{\nu,p} (B(x,r)) \approx r^d \left(1+ h_p(r)\right), \quad  r>0, \, x \in\Rd.
	\end{equation}
\end{theorem}

\begin{proof}
We will slightly abuse the notation and write $\nu(\ud x) = \nu (x) \,\ud x$. 
It suffices to validate \eqref{eq:21} for balls centered at the origin as $\Cap_{\nu,p}$ is invariant under translation. Let $B= B(0,r)$ with $r>0$. 
By the definition of $\Cap_{\nu,p}$ and Theorem \ref{thm-3-5}, we get
\begin{align*}
\Cap_{\nu,p}(B) &= \inf \{ \|f\|_{W^{\nu}_p}: f \in W_p^{\nu} \,\, \textnormal{and} \,\, B \subset \operatorname{int}(\{f\geq 1\}) \}
 \\
&\gtrsim \int_B \left(1 + h_p(|x|)\right) \ud x\\
&\gtrsim \int_{r/2<|x|<r}\left(1+ \int_{\Rd}  \left(1\wedge \frac{|h|^p}{|x|^p}\right) \nu ( h) \,\ud h \right) \ud x \\
&\gtrsim r^d \left( 1+ \int_{\Rd}  \left(1\wedge \frac{|h|^p}{r^p}\right) \nu (h)\, \ud h \right) \\
&= r^d \left(1+h_p(r)\right),
\end{align*}
which gives the lower bound estimate of \eqref{eq:21}. 

It remains to check the upper bound of \eqref{eq:21}. Choose $\Phi \in C_c^{\infty}(\Rd)$ such that
\begin{equation}\label{eq:24}
\begin{cases}
0\leq\Phi\leq 1\\
\Phi=1 \, \textnormal{on} \, B \\
\supp \Phi \subset 2B\\
\| | \nabla \Phi | \|_{\infty} \lesssim \frac{1}{r}.
\end{cases}
\end{equation} 
Then
$$
\Cap_{\nu,p} (B) \leq \|\Phi\|_{\Wp}^p = \|\Phi\|_p^p + [\Phi]_{\Wp}^p.
$$
Clearly, we have 
$$
\|\Phi\|_{p}^p = \int_{2B} |\Phi(x)|^p \, \ud x \lesssim r^d.
$$
Furthermore, the mean value theorem and the boundedness of $\Phi$ imply that
$$
|\Phi(x+h)-\Phi(x)| \lesssim 
1\ \wedge \frac{|h|}{r},
$$
which gives
\begin{align*}
[\Phi]_{\Wp}^p &\lesssim \int_{\Rd} \int_{x \in 2B \, \textnormal{or} \, x+h\in 2B}
\left( 1 \wedge \frac{|h|^p}{r^p}\right)
 \ud x \,\nu(h)\, \ud h \\
&\lesssim r^d \int_{\Rd}
\left( 1 \wedge \frac{|h|^p}{r^p}\right)
  \nu(h) \, \ud h. 
\end{align*}
Finally,
$$
\Cap_{\nu,p} (B)  \lesssim \left(1 + h_p(r)\right) r^d.
$$
\end{proof}

\begin{proposition}
Let $\nu(\ud x) = |x|^{-d} \phi(|x|)^{-1}\, \ud x$, where $\phi(r)$ is regularly varying at zero with index $\rho < d \wedge p$. Then the following holds:
	\begin{equation}\label{cap-estim-reg-var}
		\Cap_{\nu,p} (B(x,r)) \approx r^d \left(1+ h_p(r)\right), \quad r>0, \, x\in\Rd.
	\end{equation}
\end{proposition}
\begin{proof}
By Potter's theorem \cite[Theorem 1.5.6]{Bingham}, there exists $R_0>0$ such that
\begin{equation}\label{phi-scal}
\frac{1}{2}\left(\frac{r_2}{r_1}\right)^{\rho} \leq \frac{\phi(r_2)}{\phi(r_1)}
	\leq 	2 \left(\frac{r_2}{r_1}\right)^{\rho}, \quad \quad 0<r_1\leq r_2<R_0.
\end{equation}
If $\rho<1$, then $\nu$ satisfies the assumptions of Theorem \ref{21} and the claim follows. 

Let $\delta_x \coloneq \dist (x, B_{R_0}^c) = \min \{|x|, R_0-|x|\} \leq |x|$. If $\rho \in (0,d\wedge p)$, then by \cite[Theorem 5 and Remark 3.2]{DydaVahakangas}, there exists $R>0$ such that
$$
\int_{B_{R_0}} \frac{|u(x)|^p}{\phi(\delta_x)} \, \ud x\lesssim \int_{B_{R_0}} \int_{B_{R_0}\cap B(x,R\delta_x)} \frac{|u(x)-u(y)|^p}{\phi(\delta_x)\delta_x^d} \, \ud y \ud x.
$$
By \eqref{phi-scal}, $\phi(\delta_x) \approx \phi(|x|)$ and $\phi (\delta_x) \approx \phi(|x-y|)$. 
Hence
\begin{equation}\label{hardy-ineq-reg-var}
\int_{B_{R_0}} \frac{|f(x)|^p}{\phi(|x|)} \, \ud x
\lesssim 
\int_{B_{R_0}}\int_{B_{R_0}} \frac{|u(x)-u(y)|^p}{ \phi(|x-y|) |x-y|^d} \,\ud y \ud x.
\end{equation}
Note that $L(s) \approx \phi(s)^{-1}$ and $L$ satisfies the assumptions of Lemma \ref{lem:h-L}. Proceeding similarly as in the proof of Theorem \ref{thm-3-5}, we obtain
$$
\int_{B_{R_0}^c} |f(x)|^p h_p(|x|)\, \ud x \lesssim \|f\|_p^p
$$
which combined with \eqref{hardy-ineq-reg-var} gives
\begin{align*}
\int_{\Rd} |f(x)|^p h_p(|x|) \,\ud x \lesssim 
\|f\|_{\Wp}.
\end{align*}
Now, proceeding as in the proof of Theorem \ref{21} yields \eqref{cap-estim-reg-var}.
\end{proof}

\begin{example}
Let $\nu(\ud x) = |x|^{-d} \log^{\gamma}(1+|x|^{-\delta}) \,\ud x$,
$\delta,\gamma>0$ and $\delta\gamma<1$.
In Example \ref{ex-alpha-01} we obtained estimates of $L$.
By Lemma \ref{lem:h-L}, $h_p \approx L$, that is
$h_p(s) \approx \log^{\gamma}(1+s^{-\delta})$ for $s\geq 1/2$ and 
$h_p(s) \approx   \log^{\gamma+1}(1+s^{-\delta})$ for $s<1/2$.
Hence, by Theorem \ref{21}, for any ball $B$ with radius $r>0$,
$$
\Cap_{\nu,p} (B) \approx r^d \left(1+  \log^{\gamma+1}\left(1+\frac{1}{r^{\delta}}\right)\right).
$$
\end{example}

\bibliographystyle{abbrv}

\end{document}